\documentclass[12pt,reqno]{amsart}

\usepackage{amssymb}
\usepackage{graphicx}
\usepackage{amsfonts}
\usepackage{mathptmx}
\usepackage{dsfont}
\usepackage{tikz}
\usepackage{bm}
\usetikzlibrary{arrows}
\usepackage{mathtools}

\usepackage{rotating}
\usepackage{lscape}
\usepackage{hyperref}
\usepackage[mathscr]{euscript}

\usepackage[ruled,vlined,algosection,norelsize]{algorithm2e}

\usepackage{pgf, tikz}
\usepackage{color}

\definecolor{light}{gray}{0.65}

\numberwithin{equation}{section}

\def\lcm{\operatorname{lcm}}

\DeclareMathOperator{\vol}{vol}

\DeclareMathOperator{\HB}{HB}

\def\ZZ{{\mathbb Z}}

\def\QQ{{\mathbb Q}}
\def\RR{{\mathbb R}}

\newtheorem{corollary}[algocf]{Corollary}
\newtheorem{theorem}[algocf]{Theorem}
\newtheorem{proposition}[algocf]{Proposition}

\theoremstyle{definition}
\newtheorem{definition}[algocf]{Definition}
\newtheorem{remark}[algocf]{Remark}

\newtheorem{example}[algocf]{Example}
\newtheorem{question}[algocf]{Question}

\numberwithin{algocf}{section}

\def\ttt#1{\texttt{#1}}

\usepackage{color}

\begin{document}

\title{On the score sheets of a round-robin football tournament}

\author{Bogdan Ichim}

\address{Simion Stoilow Institute of Mathematics of the Romanian Academy, Research Unit 5, P.O. Box 1-764, 014700 Bucharest, Romania} \email{bogdan.ichim@imar.ro}

\author{Julio Jos\'e Moyano-Fern\'andez}

\address{Universitat Jaume I, Campus de Riu Sec, Departamento de Matem\'aticas \& Institut Universitari de Matem\`atiques i Aplicacions de Castell\'o, 12071
Caste\-ll\'on de la Plana, Spain} \email{moyano@uji.es}

\thanks{The first author was partially supported  by a grant of the Romanian National Authority for Scientific Research and Innovation,
CNCS/CCCDI -- UEFISCDI, project number PN-III-P2-2.1-PED-2016-0436, within PNCDI III. The second author was partially supported by the Spanish Government Ministerio de Econom\'ia, Industria y Competitividad (MINECO), grants MTM2012-36917-C03-03, MTM2015-65764-C3-2-P, MTM2016-81735-REDT and MTM2016-81932-REDT, as well as by Universitat Jaume I, grant P1-1B2015-02.}

\begin{abstract}
The set of (ordered) score sheets of a round-robin football tournament played
between $n$ teams together with the pointwise addition has the structure of an affine monoid. In this paper we study (using both theoretical and  computational methods)
the most important invariants of this monoid, namely the Hilbert basis, the multiplicity, the Hilbert series and the Hilbert function.
\end{abstract}

\maketitle

\section{Introduction}

Football matches are nowadays ubiquitous, stirring up deep passions in the society. Within the several procedures for selecting a winner from among many contestants, round-robin tournaments are widely extended.  A round-robin tournament is a competition in which each contestant meets all other contestants in turn. Famous examples of this kind of tournaments are the group stages of the FIFA World Cup, the UEFA European Football Championship and the UEFA Champions League (each group contains 4 teams). Goal difference (calculated as the number of goals scored minus the number of goals conceded) may be used to separate teams that are at the same level on points (according to the specific rules of a given tournament).


We consider score sheets of round-robin tournaments as tables showing the number of goals each team scores any other. Then, the problem of counting all possible score sheets (with the total number of goals fixed) of a certain tournament becomes naturally a problem for combinatorics and statistics. Moreover, we consider \emph{ordered} score sheets in the sense that the teams are ordered by the total number of goals scored. While counting score sheets in general is an easy problem, the effective counting of ordered score sheets turns out to be a quite challenging problem for which a possible approach is presented in this paper. We would like point out that the counting of score sheets is a relevant issue in the theory of weighted logic and statistical games; for instance, see the book of Epstein~\cite[Chapter 9]{E} which deals with score sheets of round-robin tournaments recording the number of points obtained by teams.

Round-robin tournaments may also be seen as directed graphs (see \cite{HM} or \cite{K}). We do not follow this approach here, our presentation is based on  monoid theory.

In a recent paper D. Zeilberger \cite{EZ} deduced a formula for counting of the number of ways for the four teams of the group stage of the FIFA World Cup to each have $r$ goals \emph{for} and $r$ goals \emph{against}. This is nothing but counting the number of $4\times 4$ magic squares with line-sums equal to $r$ under the restriction of all diagonal entries to be $0$---the problem without this restriction was treated by R.~P.~Stanley \cite[Proposition 4.6.19]{St}. He also changes the number of teams in a computational experiment, see his web-page \cite{Z}.

Inspired by these results, we present in this paper a systematic study of the score sheets. We consider the set of (ordered) score sheets of a round-robin football tournament played between $n$ teams. Since it turns out that this set has a natural structure of an affine monoid, say $\mathscr{M}_n$, we focus our attention on the most important invariants of the monoid --- that are by no means easy to compute in general. More precisely, we study the Hilbert basis, the multiplicity, the Hilbert series and the Hilbert function of $\mathscr{M}_n$.

We note that, from a practical point of view, the most important numerical invariants are the Hilbert series and the Hilbert function, since they may be used for counting the number of ordered score sheets that may result from a round-robin football tournament. However, it turns out that they are also the most difficult to compute.

This paper is structured as follows. In Section 2 we review some standard facts on rational cones and affine monoids. For details we refer the reader to Bruns and Gubeladze \cite{BG} or Bruns and Herzog \cite{BH}. Section 3 provides a short exposition about score sheets in the context of algebraic combinatorics; in particular, we introduce the monoid $\mathscr{M}_n$ of the ordered score sheets of a round-robin tournament played by $n$ teams. In Section 4 we study the Hilbert basis of $\mathscr{M}_n$ for which we give a general description in Theorem \ref{theo:HB}.

In Section 5 we focus on counting ordered score sheets with a given total number of goals i.e., on computing the Hilbert series and the Hilbert function. The Hilbert function is given by a quasipolynomial, for which we can compute its leading coefficient (i.e., the multiplicity of the monoid) in general, see Theorem \ref{theo:Mult}. Unfortunately, we are not able to present a general description of the  Hilbert series; this may be a very difficult problem (see Question \ref{quest:HS}). Using the software Normaliz \cite{Nmz}, we study several (particular, but important) cases, namely we compute the Hilbert series for tournaments with up to 7 teams. The closing Section 6 contains a report on the computational experiments done.

\section{Preliminaries}

In this section  we recall the terminology used, following Bruns and Gubeladze \cite{BG}.

\subsection{Rational cones, affine monoids and Hilbert bases} We denote by $\ZZ_+,\QQ_+$, resp. $\RR_+$ the subsets of nonnegative elements in $\ZZ,\QQ,\RR$.
Let $0\neq d \in \ZZ_{+}$ and $0\neq a\in \mathbb{Q}^{d}$. Denoting by $\langle \ ,\ \rangle$ the scalar product in $\mathbb{R}^{d}$, we consider the \emph{closed rational linear halfspaces}
$H^{+}_{a}=\{x\in
\mathbb{R}^{d}\ | \ \langle x,a\rangle \geq 0\} \ {\rm{and}} \  H^{-}_{a}=H^{+}_{-a}=\{x\in
\mathbb{R}^{d}\ | \ \langle x,a\rangle \leq  0\}.$
A subset $C\subset \RR^d$ is called a \emph{rational cone} if it is the intersection of
finitely many closed linear rational halfspaces. Equivalently, by the theorem of Minkowski--Weyl (see \cite[1.15]{BG}), $C$ is of
type $\RR_+x_1+\dots+\RR_+x_t$ with $x_i\in\QQ^d$, $i=1,\dots,t$. With this presentation, we say that $x_1,\ldots,x_t$ form a \emph{system of generators} for
$C$. If a rational cone $C$ is presented as $C=H^{+}_{a_{1}}\cap \ldots
\cap H^{+}_{a_{s}}$ such that no  $H^{+}_{a_{i}}$ can be
omitted, then we say that this is an \emph{irredundant representation} of $C$. The \emph{dimension} of a cone is the dimension of the smallest vector subspace of $\RR^d$ which contains it. If $\dim
C=d$, then the halfspaces in an irredundant representation of $C$  are uniquely determined.
A cone is \emph{pointed} if
$x,-x\in C$ implies $x=0$. In the following all cones $C$ are rational and pointed and we shall omit these attributes.

A hyperplane $H$ is called a \emph{supporting hyperplane} of a cone $C$ if $C\cap H\neq \emptyset$ and $C$ is contained in one of the closed halfspaces determined by $H$.
If $H$ is a supporting hyperplane of $C$, then $F=C\cap H$ is called a \emph{face} of $C$.
Faces with $\dim=1$ are called \emph{extreme rays} and the vectors with coprime integral components spanning them are called \emph{extreme integral generators}.

An \emph{affine monoid} $M$ is a finitely generated submonoid of the lattice $\ZZ^d$.
We say that $M$ is \emph{positive} if $x,-x\in M$ implies $x=0$.
As in \cite{BG}, we use the following terminology.

\begin{definition}
Let $M$ be a positive affine monoid. An element $x\in M$ is \emph{irreducible} if $x\neq 0$ and $x=y+z$ with $y,z\in M$ implies
$y=0$ or $z=0$. A subset $B$
of $M$ is a \emph{system of generators} if
$M=\ZZ_+B$. The unique minimal
system of generators of $M$ given by its irreducible elements
is called the \emph{Hilbert basis} and denoted by $\HB(M)$.
\end{definition}

The Hilbert basis is necessarily finite since $M$ has a finite system
of generators. Moreover, every system of generators
contains the Hilbert basis.
(One should note that sometimes a not minimal system of generators of $M$ is called a Hilbert basis of $M$.)
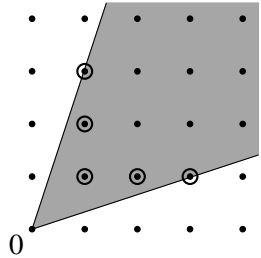
\begin{figure}[bht]\label{fig:1}
\begin{center}
	\begin{tikzpicture}[scale=0.7]
	\filldraw[light] (0,0) -- (1.411,4.3) -- (4.3,4.3) -- (4.3,1.411) -- cycle;
	\draw (0,0) -- (1.411,4.3);
	\draw (0,0) -- (4.3,1.411) node at (-0.3,-0.3){\small $0$};
	\foreach \x in {0,...,4}
	\foreach \y in {0,...,4}
	{
		\filldraw[fill=black] (\x,\y)  circle (1.5pt);
	}
	\draw[black,thick] (1,1) circle (4pt);
	\draw[black,thick] (1,2) circle (4pt);
	\draw[black,thick] (1,3) circle (4pt);
	\draw[black,thick] (2,1) circle (4pt);
    \draw[black,thick] (3,1) circle (4pt);
	\end{tikzpicture}
\end{center}
\caption{A monoid of type $C\cap\ZZ^2$ and its Hilbert basis}
\end{figure}

We are interested in sets $M$ that arise
as intersections of rational cones $C$ with the lattice $\ZZ^d$.
By Gordan's lemma (see \cite[2.9]{BG}), these sets are in fact affine monoids. Moreover, $C$ is pointed if and only if $M$ is positive.
In the following all monoids are positive, and we
omit this attribute. Let us say that $y\in M$
\emph{reduces} $x\in M$ if $y\not=0$, $x\neq y$, and $x-y\in
M$.

\begin{proposition}
Let $M$ be an affine monoid,
$B\subset M$ a system of generators, and $x\in B$. If
$x$ is reduced by some $y\in B$, then $B\setminus\{x\}$ is again a
system of generators.
\end{proposition}

\begin{proof}
Note that $B$ contains $\HB(M)$. It is enough to show that
$B\setminus\{x\}$ contains $\HB(M)$ as well. If $x-y\not=0$, then $x$ is reducible, and does not belong to $\HB(M)$.
\end{proof}

We deduce the following useful criterion.

\begin{corollary}\label{cor:HB} Let $M$ be an affine monoid and $B$ a subset of $M$ such that $0\notin B$. Then $B=\HB(M)$ if and only if the following conditions are fulfilled:
\begin{enumerate}
\item $B$ is a system of generators for $M$;
\item for every distinct $x, y \in B$, the difference $x-y \notin M$.
\end{enumerate}
\end{corollary}

\subsection{Hilbert series and functions of affine monoids} Let $C\subset \RR^d$ be a full dimensional cone (i.e. $\dim C=d$) and $M=C\cap\ZZ^d$ be the corresponding affine monoid.
A \emph{$\ZZ$-grading} of  $\ZZ^d$ is simply a linear map $\deg: \ZZ^d\to\ZZ$. If all sets $\{x\in M: \deg x=u\}$ are finite, then the \emph{Hilbert series} of $M$
with respect to the grading $\deg$ is the formal Laurent series
$$
H_M(t)=\sum_{i\in \ZZ} \#\{x\in M: \deg x=i\}t^{i}=\sum_{x\in M}t^{\deg x}.
$$
We shall assume in the following that $\deg x >0$ for all nonzero $x\in M$ (such a grading exists given the fact that $C$ is pointed)
and that there exists an $x\in\ZZ^d$ such that $\deg x=1$.
Then the Hilbert series $H_M(t)$ is the Laurent expansion of a rational function at the origin:

\begin{theorem}[Hilbert, Serre, Ehrhart, Stanley] Let $M=C\cap\ZZ^d$ as above.
	The Hilbert series of $M$ may be written in the form
	$$
	H_M(t)=\frac{R(t)}{(1-t^e)^d},\qquad R(t)\in\ZZ[t], 
	$$
	where $e$ is the $\lcm$
	of the degrees of the extreme integral generators of $C$.
\end{theorem}

Note that the above representation is not unique and that (in general) one can find denominators for $H_M(t)$ of much lower
degree than those in the theorem. This problem is discussed in Bruns, Ichim and S\"oger \cite[Section 4]{BIS}. 

An  equivalent statement can be given using the \emph{Hilbert function}
$$H(M,i)=\#\{x\in M: \deg x=i\}.$$

\begin{theorem}
There exists a quasipolynomial $Q$ with rational coefficients, degree $d-1$ and period $p$ dividing $e$ such that $H(M,i)=Q(i)$ for all $i\ge0$.
\end{theorem}

We recall that a function $Q:\ZZ\to\QQ$ is called a \emph{quasipolynomial} of
\emph{degree} $u$ if
$$
Q(i)=a_u(i)i^u+a_{u-1}(i)i^{u-1}+\cdots+a_1(i)i+a_{0}(i),
$$
where $a_{\ell}:\ZZ\to\QQ$ is a periodic function for $\ell=0,\dots,u$, and
$a_u\neq 0$. The \emph{period} of $Q$ is the smallest  positive
integer $p$ such that $a_{\ell}(i+jp)=a_{\ell}(i)$
for all $i,j\in \ZZ$ and $\ell=0,\dots,u$.

The cone $C$ together with the grading $\deg$ define the rational
polytope 
\[
\mathscr{P}=\mathscr{P}_{C}=C\cap \{x\in \RR^d:\deg x=1\}.
\]
The Hilbert series is precisely the Ehrhart series of
$\mathscr{P}$, see also \cite[6.51]{BG}.

It is not hard to show that the $p$ polynomial components of $Q$ have the same degree $d-1$ and the same leading coefficient
$a_{d-1}=\frac{\vol(\mathscr{P})}{(d-1)!},$
where $\vol(\mathscr{P})$ is the lattice normalized volume of $\mathscr{P}$ (a lattice simplex of smallest possible volume has volume $1$). 
The parameter $e(M)=\vol(\mathscr{P})=a_{d-1}(d-1)!$ is called the \emph{multiplicity} of $M$.

The reader interested 
in further details about Hilbert bases, Hilbert series and Hilbert function is referred to the books of 
Bruns and Gubeladze \cite{BG} and Bruns and Herzog \cite{BH}.

\section{Score sheets}\label{sec:scoresheets}

In this section we introduce the terminology on score sheets of round-robin football tournaments that we use in the following and we describe the (easy) general case.

\subsection{The general case}

A \emph{score sheet} $S$ of a round-robin football tournament played between $n$ teams $T_1,\dots,T_n$ looks like the following table:
$$
\begin{array}{c|c c c c}
&T_1&T_2& \cdots &T_n\\
\hline
T_1&\ast&g_{12}&\ldots&g_{1n}\\
T_2&g_{21}&\ast&\ldots&g_{2n}\\
\vdots&\vdots&&\ddots&\vdots\\
T_n&g_{n1}&g_{n2}&\ldots&\ast\\
\end{array}
$$
where $g_{ij} \in \mathbb{Z}_{+}$ and $g_{ij}$ represents the number of goals that the team $T_i$ scored the team $T_j$.
We denote by $\mathscr{S}_n$ the set of all score sheets of a round-robin football tournament played between $n$ teams.
Remark that an entry $g_{ij}$ may be any nonnegative integer.

Further, we denote by $g_i$ the total number of goals scored by $T_i$, i.e.
$$
g_i=\sum_{\substack{j=1\\i\neq j}}^{n} g_{ij}.
$$
We also set $g_S=(g_1, \ldots, g_n)$, and the $\ZZ$-grading $\deg S=|\!|S|\!|=g_1+\cdots + g_n$.

\begin{example}
If we take the results of the group H in the South Africa FIFA World Cup 2010, the corresponding score table alphabetically ordered is
\small
$$
\begin{array}{c|c c c c}
&\mathrm{Chile}&\mathrm{Honduras}& \mathrm{Spain} &\mathrm{Switzerland}\\
\hline
\mathrm{Chile}&\ast&1&1&1\\
\mathrm{Honduras}&0&\ast&0&0\\
\mathrm{Spain}&2&2&\ast&0\\
\mathrm{Switzerland}&0&0&1&\ast
\end{array}
$$
\normalsize
\end{example}

It is easy to see that the pointwise addition of two score sheets $S_1, S_2\in \mathscr{S}_n $ is again an element of $\mathscr{S}_n$.
Thus $\mathscr{S}_n$ has a monoid structure and we can identify
$$
\mathscr{S}_n\backsimeq \ZZ_{+}^{n^2-n}.
$$
Obviously the monoid $\mathscr{S}_n$ has the same Hilbert basis as the monoid $\ZZ_{+}^{n^2-n}$. It is given by the set of the score sheets with the following properties:
\begin{enumerate}
\item There exist $1\le i,j\le n$ such that the team $T_i$ has scored exactly one goal against team $T_j$ and no goal against other team;
\item For all $k\not=i$ the teams $T_{k}$ have scored no goal against other team.
\end{enumerate}

Further, one may count scoring tables the same way that elements in $\ZZ_{+}^{n^2-n}$ are counted.
Set $\deg: \ZZ_{+}^{n^2-n}\to\ZZ$ to be the sum of the components, that is
$$\deg(S)=|\!|S|\!|.$$
Then we get that the number of score sheets with total number of goals $G$, i.~e.
$$
\#\{S\in \mathscr{S}_n: |\!|S|\!|=G\},
$$
is given by the Hilbert function $H:\mathbb{Z}_{+} \to \mathbb{Z}_{+}$ of the monoid $\ZZ_{+}^{n^2-n}$. Moreover, the monoids have the same multiplicity.
Using well known facts about the monoid $\ZZ_{+}^{n^2-n}$, we easily obtain the following theorem:

\begin{theorem}
Let $G=\sum_{\substack{i,j=1\\i\neq j}}^{n} g_{ij}$ be the total number of goals in a round-robin football tournament played between $n$ teams. Then  for the number of scoring tables we obtain:
\begin{align*}
& {\bf (1)} \quad \#\{S\in \mathscr{S}_n: |\!|S|\!|=G\}={n^2-n+G-1\choose G};\\
& {\bf (2)} \quad \#\{S\in \mathscr{S}_n: |\!|S|\!|\le G\}={n^2-n+G\choose G}.
\end{align*}
Moreover, $e(\mathscr{S}_n)=1$. This may be further interpreted as
$$
\lim_{G\rightarrow\infty}\frac{\#\{S\in \mathscr{S}_n: |\!|S|\!|=G\}}{G^{n^2-n-1}}=\frac{1}{(n^2-n-1)!}.
$$
\end{theorem}

\subsection{Ordered score sheets}

We say that a score sheet $S\in \mathscr{S}_n$ of a round-robin football tournament played between $n$ teams is an \emph{ordered score sheet} if it satisfies the following inequalities:
$$
\sum_{\substack{j=1\\i\neq j}}^{n} g_{i,j} \geq \sum_{\substack{j=1\\i+1\neq j}}^{n} g_{i+1,j} \ \ \ \forall i=1,\dots,n-1,
$$
or, equivalently, the inequalities $g_1\ge g_2 \ge \cdots\ge g_n$. We remark that this is in fact the form in which most score sheets are presented.
It is easy to see that the set of the ordered score sheets of a round-robin football tournament played between $n$ teams is a submonoid of $\mathscr{S}_n$,
which we denote in the following by  $\mathscr{M}_n$.

\begin{remark} There are exactly $n!$ total orders that may be defined on the set $g_1,\ldots, g_n$.
We have chosen the natural one in order to define $\mathscr{M}_n$.
\end{remark}

\begin{example}\label{ex:orderedsheet}
The ordered score sheet of the group H in the FIFA World Cup 2010 is the following:
\small
$$
\begin{array}{c|c c c c}
&\mathrm{Spain}&\mathrm{Chile}& \mathrm{Switzerland} &\mathrm{Honduras}\\
\hline
\mathrm{Spain}&\ast&2&0&2\\
\mathrm{Chile}&1&\ast&1&1\\
\mathrm{Switzerland}&1&0&\ast&0\\
\mathrm{Honduras}&0&0&0&\ast
\end{array}
$$
\normalsize
\end{example}


\section{The Hilbert basis of the monoid $\mathscr{M}_n$}

In this section we present a general description of the Hilbert basis of the monoid $\mathscr{M}_n$ of the ordered score sheets of a round-robin football tournament played between $n$ teams.

\begin{theorem}\label{theo:HB}
Let $i\in \ZZ_+$, with $1\le i \le n$. The Hilbert basis of the monoid $\mathscr{M}_n$ is given by the set of ordered score sheets with the following properties:
\begin{enumerate}
\item The teams $T_1, \ldots , T_i$ have scored exactly one goal against any other team;
\item The teams $T_{i+1}, \ldots, T_n$ have scored no goal against other team.
\end{enumerate}
Moreover, the number of elements in  the Hilbert basis is given by the following formula:
\[
\#\HB(\mathscr{M}_n)=\sum_{i=1}^n (n-1)^{i}.
\]
\end{theorem}

\begin{proof}
Let $B_n$ be the subset of $\mathscr{M}_n$ containing all score sheets $S$ such that  the teams $T_1, \ldots , T_i$ have scored exactly one goal against any other team, and the teams $T_{i+1}, \ldots, T_n$ have scored no goal against other team for $i=1,\ldots,n$; or, in other words, the subset of the score sheets $S$ where the vector $g_S$ is one of the following:
\begin{align*}
p_1&=(1,0,0,\ldots , 0)\\
p_2&=(1,1,0,\ldots , 0)\\
p_3&=(1,1,1,\ldots , 0)\\
\dots&\ldots \ldots \ldots \ldots \ldots \dots  \\
p_n&=(1,1,1,\ldots , 1)
\end{align*}
We have to show that $B_n$ is the Hilbert basis of $\mathscr{M}_n$. Using Corollary \ref{cor:HB}, we have to prove that the following is true:
\begin{enumerate}
\item each element in $\mathscr{M}_n$ can be written as a linear combination with positive coefficients of elements of $B_n$;
\item for every distinct $A, B \in B_n$, the difference $A-B \notin \mathscr{M}_n$.
\end{enumerate}

\noindent \textbf{(1)} Let $S \in \mathscr{M}_n$, then there is a unique writing
\[
g_S=\sum_{i=1}^{n} a_ip_i \ \ \mbox{~with~\ } a_i \in \mathbb{Z}_{+}.
\]
We may assume that $a_r \neq 0$; this means that for $i=1, \ldots, r$ we may associate $j_i$ such that $g_{i,j_i} \neq 0$. We consider $H \in B_n$ which has precisely $1$ at the position $(i,j_i)$, and $0$ otherwise. Then $S-H \in \mathscr{M}_n$. Since $|\!|S-H|\!| < |\!|S|\!|$, the conclusion follows by induction.
\medskip

\noindent \textbf{(2)} Let $A, B \in B_n$, and consider the associated vectors $g_A$ and $g_B$. Then we have that $g_A,\ g_B\in \{p_1,\ldots, p_n\}$ and we may identify $g_A=p_i$ and $g_B=p_j$. It follows that $g_{A-B}=g_A-g_B=p_i-p_j.$
Let us assume that $A-B \in \mathscr{M}_n$. If $i<j$, then $p_i-p_j\not\in \ZZ_{+}^{n}$, so $A-B$ is not a score sheet. If $i>j$, then the first entry of $g_{A-B}$ is $0$, while the $i$-th entry is $1$, so $A-B$ is not an ordered score sheet. The only possibility left is that $g_A=g_B$, so $g_{A-B}=(0,\ldots,0)$. On each line both $A$ and $B$ have exactly one non zero entry, therefore we deduce that $A$ and $B$ coincide entrywise.
\end{proof}

\begin{example}
For $n=3$ the Hilbert basis of the monoid consists of 14 elements, namely:
\small
$$
\begin{array}{cccc}
\begin{array}{c|c c c}
&T_1&T_2 &T_3\\
\hline
T_1&\ast&1&0\\
T_2&0&\ast &0\\
T_3&0&0&\ast
\end{array}
&
\begin{array}{c|c c c}
&T_1&T_2 &T_3\\
\hline
T_1&\ast&0&1\\
T_2&0&\ast &0\\
T_3&0&0&\ast
\end{array}
&
\begin{array}{c|c c c}
&T_1&T_2 &T_3\\
\hline
T_1&\ast&1&0\\
T_2&1&\ast &0\\
T_3&0&0&\ast
\end{array}
&
\begin{array}{c|c c c}
&T_1&T_2 &T_3\\
\hline
T_1&\ast&1&0\\
T_2&0&\ast &1\\
T_3&0&0&\ast
\end{array}
\end{array}
$$
$$
\begin{array}{cccc}
\begin{array}{c|c c c}
&T_1&T_2 &T_3\\
\hline
T_1&\ast&0&1\\
T_2&1&\ast &0\\
T_3&0&0&\ast
\end{array}
&
\begin{array}{c|c c c}
&T_1&T_2 &T_3\\
\hline
T_1&\ast&0&1\\
T_2&0&\ast &1\\
T_3&0&0&\ast
\end{array}
&
\begin{array}{c|c c c}

&T_1&T_2 &T_3\\
\hline
T_1&\ast&1&0\\
T_2&1&\ast &0\\
T_3&1&0&\ast
\end{array}
&
\begin{array}{c|c c c}
&T_1&T_2 &T_3\\
\hline
T_1&\ast&1&0\\
T_2&1&\ast &0\\
T_3&0&1&\ast
\end{array}
\end{array}
$$
$$
\begin{array}{cccc}
\begin{array}{c|c c c}
&T_1&T_2 &T_3\\
\hline
T_1&\ast&1&0\\
T_2&0&\ast &1\\
T_3&1&0&\ast
\end{array}
&
\begin{array}{c|c c c}
&T_1&T_2 &T_3\\
\hline
T_1&\ast&1&0\\
T_2&0&\ast &1\\
T_3&0&1&\ast
\end{array}
&
\begin{array}{c|c c c}
&T_1&T_2 &T_3\\
\hline
T_1&\ast&0&1\\
T_2&1&\ast &0\\
T_3&1&0&\ast
\end{array}
&
\begin{array}{c|c c c}
&T_1&T_2 &T_3\\
\hline
T_1&\ast&0&1\\
T_2&1&\ast &0\\
T_3&0&1&\ast
\end{array}
\end{array}
$$
$$
\begin{array}{cc}
\begin{array}{c|c c c}
&T_1&T_2 &T_3\\
\hline
T_1&\ast&0&1\\
T_2&0&\ast &1\\
T_3&1&0&\ast
\end{array}
&
\begin{array}{c|c c c}
&T_1&T_2 &T_3\\
\hline
T_1&\ast&0&1\\
T_2&0&\ast &1\\
T_3&0&1&\ast
\end{array}
\end{array}
$$
\normalsize
\end{example}

\begin{example}
For $n=4$ there are 120 elements in the Hilbert basis. We do not present all of them  due to space constraints, but we would still like to point out the following:
If we consider the ordered score sheet in Example \ref{ex:orderedsheet} we may easily compute the following decomposition:
\small
$$
\begin{array}{ccccc}
\begin{array}{c|c c c c}
&\mathrm{Sp}&\mathrm{Ch}& \mathrm{Sw} &\mathrm{Hon}\\
\hline
\mathrm{Sp}&\ast&2&0&2\\
\mathrm{Ch}&1&\ast&1&1\\
\mathrm{Sw}&1&0&\ast&0\\
\mathrm{Hon}&0&0&0&\ast
\end{array}
&
=
&
\begin{array}{|c c c c|}
\hline
\ast&0&0&1\\
0&\ast&0&1\\
1&0&\ast&0\\
0&0&0&\ast\\
\hline
\end{array}
&
+
&
\begin{array}{|c c c c|}
\hline
\ast&1&0&0\\
1&\ast&0&0\\
0&0&\ast&0\\
0&0&0&\ast\\
\hline
\end{array}
\\
&&&&
\\
&
+
&
\begin{array}{|c c c c|}
\hline
\ast&1&0&0\\
0&\ast&1&0\\
0&0&\ast&0\\
0&0&0&\ast \\
\hline
\end{array}
&
+
&
\begin{array}{|c c c c|}
\hline
\ast&0&0&1\\
0&\ast&0&0\\
0&0&\ast&0\\
0&0&0&\ast\\
\hline
\end{array}
\end{array}
$$
\normalsize
\end{example}

We close the section with a the following Corollary, whose proof was pointed out by Winfried Bruns.

\begin{corollary}\label{cor:extremerays}
The Hilbert basis of $\mathscr{M}_n$ equals the set of extreme integral generators of the cone $C_n = \RR_+\mathscr{M}_n$.
\end{corollary}

\begin{proof}
We first remark that the inclusion ``$\supseteq$'' is always true. We have to show the  inclusion ``$\subseteq$''. Let $S\in\HB(\mathscr{M}_n)$ be an element of the Hilbert basis and suppose it is not an extreme integral generator.
Then it may be written as a linear combination  of the extreme integral generators with nonnegative \emph{rational} coefficients.  (By the first remark, this is in fact a linear combination of Hilbert basis elements.) There exists $1\le k\le n$ such that the element $S$ has exactly one entry $g^S_{ij_i}=1$ for each $1\le i\le k$ and no other nonzero entry. Any summand in the linear combination with nonzero coefficient can have nonzero entries only where $S$ has them. Moreover, there must be one summand $H$ with $g^H_{kj_k}=1$.  The only element in the Hilbert
 basis that satisfies these two conditions is $S$ itself! So it
 must appear in the linear combination, and this is a contradiction.
\end{proof}


\section{The multiplicity and the Hilbert series of the monoid $\mathscr{M}_n$}

This section will be concerned with the computation of two invariants of the affine monoid of the ordered score sheets of a round-robin football tournament played by $n$ teams, namely its multiplicity and Hilbert series, deriving interesting formulas for explicit calculations.

\subsection{The multiplicity of the monoid $\mathscr{M}_n$} A general formula for the multiplicity of $\mathscr{M}_n$ reads off as follows.

\begin{theorem}\label{theo:Mult}
The  multiplicity of the monoid $\mathscr{M}_n$ is given by the following formula:
\[
e(\mathscr{M}_n)=\frac{1}{n!}.
\]
We deduce that all polynomial components of quasipolynomial representing the Hilbert function of $\mathscr{M}_n$ have the same leading coefficient:
$$
\frac{1}{n!(n^2-n-1)!}.
$$
This may be further interpreted as
$$
\lim_{G\rightarrow\infty}\frac{\#\{S\in \mathscr{M}_n: |\!|S|\!|=G\}}{G^{n^2-n-1}}=\frac{1}{n!(n^2-n-1)!}.
$$
\end{theorem}

\begin{proof} There are exactly $n!$ total orders that may be defined on the set $g_1,\ldots, g_n$.
To each total order $\le_k$ we can attach a monoid of ordered score sheets $\mathscr{M}^k_n$ analogous to the way $\mathscr{M}_n$ was defined. It is clear that
$$
\mathscr{S}_n=\bigcup_{k=1}^{n!}\mathscr{M}^k_n.
$$

Let $C^k = \RR_+\mathscr{M}^k_n\subset\RR^{n^2-n}$ be the
cone generated by $\mathscr{M}^k_n$. It is clear that $\RR_{+}^{n^2-n}=\bigcup_{k=1}^{n!}C^k$. Consider two cones $C^{k_1}$ and $C^{k_2}$. Then there exist $1\le i,j\le n$ such that $g_{i}\le_{k_1} g_{j}$ and $g_{i}\ge_{k_2} g_{j}$. Therefore
$$
C^{k_1}\bigcap C^{k_2}\subset H_{g_i=g_j}.
$$
Elementary reasonings of measure theory show that, for all $1\le i,j\le n$,
$$\vol(\mathscr{P}_{C^{k_1}\bigcap C^{k_2}})=0.$$
By the inclusion-exclusion principle we obtain
\begin{align*}
\vol(\mathscr{P}_{\RR_{+}^{n^2-n}}) & =\sum_{k=1}^{n!} \vol(\mathscr{P}_{C^{k}})-\sum_{k_1\neq k_2}\vol(\mathscr{P}_{C^{k_1}\bigcap C^{k_2}})+\sum_{k_1\neq k_2\neq k_3\neq k_1}\vol(\mathscr{P}_{C^{k_1}\bigcap C^{k_2}\bigcap C^{k_3}})-\ldots\\
& =\sum_{k=1}^{n!} \vol(\mathscr{P}_{C^{k}}).
\end{align*}

Finally we deduce that
$$
n!e(\mathscr{M}_n)=\sum_{i=1}^{n!}e(\mathscr{M}^k_n)=e(\mathscr{S}_n)=1,
$$
which in turn implies the claimed formula.
\end{proof}


\subsection{Triangulations}\label{sec:triangulation} The study of the triangulations of a convex polytope is one of the most important research
areas in the theory of convex polytopes. The content of this subsection was inspired by a remark of Winfried Bruns.

Apart from the $\ZZ$-grading given by $\deg S=|\!|S|\!|=g_1+\cdots + g_n$ considered in Section \ref{sec:scoresheets}, there is another $\ZZ$-grading that is interesting, namely
$$\deg_1 S=|\!|S|\!|_1=g_1.$$
Then, according to Corollary \ref{cor:extremerays},  every extreme integral generator  has degree $1$.
In other words: the monoid $\mathscr{M}_n$ is the Ehrhart
monoid of an integral polytope $\mathscr{P}_n$ (with vertices the extreme integral generators). Since $\mathscr{P}_n$ contains the Hilbert basis of $\mathscr{M}_n$, $\mathscr{P}_n$ is \emph{integrally closed}. A sufficient condition for this to happen is that $\mathscr{P}_n$ has a \emph{unimodular triangulation}, i.e all simplices in the triangulation have the smallest possible volume 1. In general integral polytopes do not admit a unimodular triangulation, but the ones that admit such a triangulation are important for integer programming. This is indeed the case here, since $\mathscr{P}_n$ is a \emph{compressed} polytope. The class of compressed polytopes was introduced by Stanley \cite{S1}.

\begin{proposition}\label{prop:compressed}
$\mathscr{P}_n$ is a compressed polytope. As a consequence, every pulling triangulation of $\mathscr{P}_n$ is unimodular.
\end{proposition}

\begin{proof}
The $(0,1)$-polytope $\mathscr{P}_n$ consists of all solutions of the system of linear inequalities
\begin{align*}
g_1=1;& \\
0\le g_{i-1}-g_{i} \le 1& \ \ \ \forall \ i=2,\dots,n;\\
0\le g_{ij} \le 1& \ \ \ \forall \ 1\le i\not=j\le n.
\end{align*}

It follows directly from Ohsugi and Hibi \cite[Theorem 1.1]{OH} that the polytope $\mathscr{P}_n$ is compressed.
\end{proof}

Note that, on the one hand, a compressed polytope is not automatically a \emph{totally unimodular} polytope (i.e. a polytope with all triangulations unimodular).
On the other hand, it is true that a totally unimodular polytope is compressed. We will come back to this in Remark \ref{rem:tri}.


\subsection{The Hilbert series of the monoid $\mathscr{M}_n$}

The Hilbert series and quasipolynomials for the monoid of ordered score sheets of tournaments among 2, 3 and 4 teams may be easily computed on a normal computer using the software Normaliz \cite{Nmz}. Let us start with the simplest case $n=2$:

\begin{proposition}
The number of ordered score sheets $Q(G)$ in a two-team game with given number of goals $G$ is
$$
Q(G)=\left \{ \begin{array}{cc}
\frac{G+2}{2}& \mbox{if~} G ~\mbox{~even;}\\
\frac{G+1}{2}& \mbox{if~} G ~\mbox{~odd.}
\end{array}
\right.
$$
Therefore, $Q(G)$ is a quasipolynomial of degree $1$ and period $2$.
\end{proposition}

For $n=3$ we obtain:

\begin{proposition}
Let $Q(G)$ denote the number of ordered score sheets with given number of goals $G$. Then, $Q(G)$ is a quasipolynomial
of degree $5$ and period $6$. Further, the Hilbert series $\sum_{G=0}^{\infty} Q(G)t^G$ is
$$
\frac{1+ 4t^2 +5t^3 +5t^4+ 8t^5+ 22t^6 -3t^7+ 25t^8+ 12t^9+ 6t^{10}+ 9t^{11}+ 12t^{12} -4t^{13}+ 6t^{14}}
{(1-t)^2(1-t^3)(1-t^6)^3}.$$ Thus, in terms of a quasipolynomial formula we have
$$
Q(G)=\left \{ \begin{array}{ll}
1+\frac{241}{180}G  +\frac{103}{144}G^2 + \frac{125}{648} G^3 + \frac{5}{192} G^4 + \frac{1}{720} G^5& \mbox{~if~} G \equiv 0 \mod 6,\\
\\
\frac{77}{192}+\frac{589}{720}G  +\frac{55}{96}G^2 + \frac{13}{72} G^3 + \frac{5}{192} G^4 + \frac{1}{720}  G^5& \mbox{~if~} G \equiv 1 \mod 6,\\
\\
\frac{2}{3}+\frac{181}{180}G  +\frac{29}{48}G^2 + \frac{13}{72} G^3 + \frac{5}{192} G^4 + \frac{1}{720}  G^5& \mbox{~if~} G \equiv 2 \mod 6,\\
\\
\frac{47}{64}+\frac{829}{720}G  +\frac{197}{288}G^2 + \frac{125}{648} G^3 + \frac{5}{192} G^4 + \frac{1}{720}  G^5& \mbox{~if~} G \equiv 3 \mod 6,\\
\\
\frac{2}{3}+\frac{181}{180}G  +\frac{29}{48}G^2 + \frac{13}{72} G^3 + \frac{5}{192} G^4 + \frac{1}{720}  G^5& \mbox{~if~} G \equiv 4 \mod 6,\\
\\
\frac{77}{192}+\frac{589}{720}G  +\frac{55}{96}G^2 + \frac{13}{72} G^3 + \frac{5}{192} G^4 + \frac{1}{720}  G^5& \mbox{~if~} G \equiv 5 \mod 6.
\end{array}
\right.
$$
\end{proposition}
\medskip

The case $n=4$ involves already big numbers:

\begin{proposition}
Let $Q(G)$ denote the number of ordered score sheets with given number of goals $G$. Then, $Q(G)$ is a quasipolynomial
of degree $11$ and period $12$. Moreover, we computed
$$\sum_{G=0}^{\infty} Q(G)t^G =
\frac{R(t)}
{(1-t)^2(1-t^2)(1-t^4)^2(1-t^{12})^{7}},
$$
where $R(t)$ equals
\begin{align*}
&1+ t+ 9t^2+ 27t^3+ 109t^4+ 79t^5+ 511t^6+ 619t^7+ 1948t^8+ 1517t^9+ 5426t^{10}\\
&+ 5233t^{11}+ 16348t^{12}+ 8123t^{13}+ 30635t^{14}+ 31085t^{15}+ 71333t^{16}+ 29381t^{17}\\ &+ 129582t^{18}+100425t^{19}+ 226133t^{20}+ 106469t^{21}+ 352563t^{22}+ 255195t^{23}\\ &+ 602325t^{24}+205146t^{25}+ 774351t^{26}+ 600016t^{27}+ 1136341t^{28}+ 348198t^{29}\\
&+ 1502118t^{30}+ 934180t^{31}+1820773t^{32}+ 644194t^{33}+ 2127040t^{34}+ 1295458t^{35}\\ &+ 2619417t^{36}+646956t^{37}+ 2642866t^{38}+ 1711900t^{39}+ 2797833t^{40}+ 648858t^{41}\\ &+ 2966496t^{42}+1529130t^{43}+2665305t^{44}+ 721796t^{45}+ 2465066t^{46}+ 1286628t^{47}\\ &+ 2281092t^{48}+375173t^{49}+1848683t^{50}+ 1014611t^{51}+ 1414748t^{52}+ 226313t^{53}\\ &+ 1226575t^{54}+509533t^{55}+792641t^{56}+ 152129t^{57}+ 572326t^{58}+ 249805t^{59}\\ &+ 381577t^{60}+25597t^{61}+239179t^{62}+ 107619t^{63}+ 117396t^{64}+ 7537t^{65}\\
&+ 81226t^{66}+ 23469t^{67} +31248t^{68}+ 2811t^{69}+ 15627t^{70}+ 5037t^{71}+ 5928t^{72} \\ &-720t^{73}+ 2325t^{74}+738t^{75}+288t^{76} -90 t^{77}+180t^{78}.
\end{align*}

In addition, we have that the quasi-polynomial $Q(G)$ equals the $i$-th row of the matrix given by the product $A\cdot (G^0, G, G^2, \ldots , G^{11})^{T}$ if $G \equiv i \mod 12$, where the matrix $A$ is written by typographical reasons as
$
A=\left(
\begin{array}{c|c}
A_1 & A_3 \\
\hline
A_2 & A_4
\end{array}
\right ),
$
and the submatrices $A_i$ are precisely

\footnotesize
$$
A_1=\left(
\begin{array}{cccccc}
1 & \frac{10631}{6160} & \frac{1640537}{1209600} & \frac{4678633}{7257600} & \frac{9526327}{46448640} & \frac{18987713}{418037760} \\
&&&&&\\
\frac{812827273}{1934917632} &\frac{9999191849}{10346434560} &\frac{5723257663}{6270566400} &\frac{490622483}{1007769600} &\frac{629685829}{3762339840} &\frac{684923}{17418240}\\
&&&&&\\
\frac{103759}{236196} &\frac{38833721}{40415760} &\frac{30392849}{32659200} &\frac{454020097}{881798400} &\frac{169838633}{940584960} &\frac{26580107}{627056640} \\
&&&&&\\
\frac{18657}{32768} &\frac{5144987}{4730880} &\frac{72507023}{77414400}  &\frac{2006881}{4147200} &\frac{7703989}{46448640} &\frac{2045809}{52254720} \\
&&&&&\\
\frac{50234}{59049} &\frac{64855871}{40415760} &\frac{130542697}{97977600} &\frac{1142071019}{1763596800} &\frac{777295207}{3762339840} &\frac{6353131}{139345920} \\
&&&&&\\
\frac{500482697}{1934917632} &\frac{7217434409}{10346434560} &\frac{1574201621}{2090188800}  &\frac{445464083}{1007769600} &\frac{605457989}{3762339840} &\frac{6092627}{156764160}
\end{array}
\right )
$$

$$
A_2=\left(
\begin{array}{cccccc}
\frac{3}{4} &\frac{8321}{6160} &\frac{1347587}{1209600} &\frac{2020379}{3628800} &\frac{2154073}{11612160} &\frac{8919769}{209018880} \\
&&&&&\\
\frac{812827273}{1934917632} &\frac{9999191849}{10346434560} &\frac{5723257663}{6270566400} &\frac{490622483}{1007769600} &\frac{629685829}{3762339840} &\frac{684923}{17418240}\\
&&&&&\\
\frac{40702}{59049} &\frac{53989631}{40415760} &\frac{38302499}{32659200} &\frac{1063043819}{1763596800} &\frac{753067367}{3762339840} &\frac{56604739}{1254113280}\\
&&&&&\\
\frac{18657}{32768} &\frac{5144987}{4730880} &\frac{72507023}{77414400} &\frac{2006881}{4147200} &\frac{7703989}{46448640} &\frac{2045809}{52254720} \\
&&&&&\\
\frac{141887}{236196} &\frac{49699961}{40415760} &\frac{106813747}{97977600} &\frac{493533697}{881798400} &\frac{175895593}{940584960} &\frac{2985203}{69672960} \\
&&&&&\\
\frac{500482697}{1934917632} &\frac{7217434409}{10346434560} &\frac{1574201621}{2090188800} &\frac{445464083}{1007769600} &\frac{605457989}{3762339840} &\frac{6092627}{156764160}
\end{array}
\right )
$$

$$
A_3=\left(
\begin{array}{cccccc}
 \frac{50777501}{7166361600} & \frac{233250653}{300987187200} & \frac{18663361}{321052999680} & \frac{771073}{270888468480} & \frac{19}{232243200} & \frac{1}{958003200}\\
&&&&&\\
\frac{12980747}{2015539200} &\frac{430729}{587865600} &\frac{141929}{2508226560} &\frac{3413}{1209323520} &\frac{19}{232243200} &\frac{1}{958003200}\\
&&&&&\\
\frac{27583349}{4031078400} &\frac{14362883}{18811699200}  &\frac{290419}{5016453120} &\frac{771073}{270888468480}  &\frac{19}{232243200} &\frac{1}{958003200}\\
&&&&&\\
\frac{1440883}{223948800}  &\frac{430729}{587865600}&\frac{141929}{2508226560} &\frac{3413}{1209323520} &\frac{19}{232243200} &\frac{1}{958003200}\\
&&&&&\\
\frac{457407109}{64497254400} &\frac{233250653}{300987187200} &\frac{18663361}{321052999680} &\frac{771073}{270888468480} &\frac{19}{232243200} &\frac{1}{958003200}\\
&&&&&\\
\frac{12955147}{2015539200} &\frac{430729}{587865600}&\frac{141929}{2508226560} &\frac{3413}{1209323520} &\frac{19}{232243200} &\frac{1}{958003200}
\end{array}
\right )
$$

$$
A_4=\left(
\begin{array}{cccccc}
\frac{3067661}{447897600} &\frac{14362883}{18811699200}  &\frac{290419}{5016453120} &\frac{771073}{270888468480} &\frac{19}{232243200} &\frac{1}{958003200} \\
&&&&&\\
\frac{12980747}{2015539200} &\frac{430729}{587865600}&\frac{141929}{2508226560} &\frac{3413}{1209323520} &\frac{19}{232243200} &\frac{1}{958003200}\\
&&&&&\\
\frac{456587909}{64497254400} &\frac{233250653}{300987187200} &\frac{18663361}{321052999680} &\frac{771073}{270888468480} &\frac{19}{232243200} &\frac{1}{958003200}\\
&&&&&\\
\frac{1440883}{223948800} &\frac{430729}{587865600}&\frac{141929}{2508226560} &\frac{3413}{1209323520} &\frac{19}{232243200} &\frac{1}{958003200}\\
&&&&&\\
\frac{27634549}{4031078400} &\frac{14362883}{18811699200} &\frac{290419}{5016453120}&\frac{771073}{270888468480} &\frac{19}{232243200} &\frac{1}{958003200}\\
&&&&&\\
\frac{12955147}{2015539200} &\frac{430729}{587865600} &\frac{141929}{2508226560} &\frac{3413}{1209323520} &\frac{19}{232243200} &\frac{1}{958003200}
\end{array}
\right )
$$
\normalsize
\end{proposition}
\medskip

The cases $n=5,6,7$ may also be computed using Normaliz \cite{Nmz} or its offspring NmzIntegrate \cite{BS}. A quite powerful computer should be used for these computations.  The results obtained are summarized in the following:

\begin{proposition}\label{prop:567}
Let $Q(G)$ denote the number of ordered score sheets with given number of goals $G$. Then
\begin{enumerate}
\item For $n=5$ it turns out that $Q(G)$ is a quasipolynomial
of degree $19$ and period $60$. The Hilbert series $\sum_{G=0}^{\infty} Q(G)t^G$ equals
$$
\frac{1+ 16t^2+ 40t^3+ 276t^4 +898t^5+\cdots + 134400 t^{652}
-67200 t^{653}+28000 t^{654} }{(1-t)^{4}(1-t^5)^{2}(1-t^{10})(1-t^{20})^{3}(1-t^{60})^{10}}.
$$
\item For $n=6$ it turns out that $Q(G)$ is a quasipolynomial
of degree $29$ and period $60$. Moreover, we computed that the Hilbert series $\sum_{G=0}^{\infty} Q(G)t^G$ equals
$$
\frac{1+ 2t+ 27t^2+ 127t^3 +852t^4+ \cdots + 60637500t^{1128} -30870000t^{1129}+ 23152500t^{1130}}{(1-t)^{3}(1-t^{2})(1-t^{6})^{5}(1-t^{30})^{4}(1-t^{60})^{17}}.
$$
\item For $n=7$ we get that $Q(G)$ is a quasipolynomial
of degree $41$ and period $420$. Further, we computed that the Hilbert series $\sum_{G=0}^{\infty} Q(G)t^G$ equals
$$
\frac{1+ 36t^2+ 126t^3 +1317t^4+ \cdots  -443603381760t^{10099}+ 103507455744t^{10100}}{(1-t)^{6}(1-t^{7})^{3}(1-t^{14})(1-t^{42})^{6}(1-t^{210})^{5}(1-t^{420})^{21}}.
$$
\end{enumerate}
\end{proposition}
Note that the numerators of Hilbert series presented as rational functions in Proposition \ref{prop:567} are not written entirely by space reasons, since they are built up with too many and big coefficients.

We would like to close this section by raising the following open question:

\begin{question}\label{quest:HS}
Does there exist a formula for the expression of the Hilbert series of $\mathscr{M}_n$  (for any $n$)?
\end{question}


\section{Computational experiments}\label{section:computational}

The results presented in the previous sections were first conjectured by extensive computational experiments.
In this section we document these experiments, in the hope that the data obtained may be useful for readers interested in running similar experiments.

For the experiments we have used the software Normaliz \cite{Nmz}, together with the extension NmzIntegrate \cite{BS} and the graphical interface jNormaliz \cite{AI}. For the algorithms implemented in Normaliz, we recommend the reader to see Bruns and Koch \cite{BK}, Bruns and Ichim \cite{BI}, Bruns, Hemmecke, Ichim, K\"oppe and S\"oger \cite{BHIKS}, Bruns and S\"oger \cite{BS}, and Bruns, Ichim, and S\"oger \cite{BIS}. All computations were run on a Dell PowerEdge R910 with 4 Intel
Xeon E7540 (a total of 24 cores running at 2 GHz), 128 GB of
RAM and a hard disk of 500 GB. In parallelized computations we
have limited the number of threads used to $20$. In
Tables \ref{times1} and \ref{times2}
serial execution is indicated by \ttt{1x}, whereas \ttt{20x} indicates parallel
execution with a maximum of $20$ threads. By \ttt{nxn} we denote the monoid $\mathscr{M}_n$.
All computation times are measured using version 3.0 of Normaliz.

\subsection{Hilbert bases computations using Normaliz}

We have run experiments using both the primal (see \cite{BI} and \cite{BHIKS}) and the dual algorithm of Normaliz (see \cite{BI}).
Table \ref{times1} contains the computation times for the
Hilbert bases of the \ttt{nxn} ordered score sheets that we have obtained in these experiments.
The command line option \ttt{-N} indicates the usage of the primal algorithm, and \ttt{-d} indicates the usage of the dual algorithm for Hilbert bases.

\begin{table}[hbt]
\centering
\begin{small}
\begin{tabular}{|r|r|r|r|r|r|}\hline
\rule[-0.1ex]{0ex}{2.5ex}Input &\ttt{Nmz -N 1x}&\ttt{Nmz -N 20x}&\ttt{Nmz  -d 1x}&\ttt{Nmz -d 20x}\\ \hline
\strut \ttt{3x3}               & 0.003 s       & 0.014 s        & 0.003 s        & 0.014 s\\ \hline
\strut \ttt{4x4}               & 0.017 s       & 0.052 s        & 0.012 s        & 0.024 s\\ \hline
\strut \ttt{5x5}               & 0.324 s       & 0.712 s        & 0.206 s        & 0.222 s\\ \hline
\strut \ttt{6x6}               & 16.973 s      & 19.296 s       & 5.888 s        & 6.603 s\\ \hline
\strut \ttt{7x7}               & 39:40 m       & 54:41 m        & 4:06 m         & 3:39 m \\ \hline
\strut \ttt{8x8}               & --            & 404:26:28 h   & 16:51:42 h     & 3:51:22 h \\ \hline
\end{tabular}
\vspace*{2ex} \caption{Computation times for Hilbert
bases}\label{times1}
\end{small}
\end{table}

We discuss the observations made during this experiment in the following remark.

\begin{remark} (a) As can be seen from the table, the dual algorithm is better suited for computing the Hilbert bases of this particular family of monoids. This confirms the empirical conclusion that the dual algorithm is faster when the numbers of support hyperplanes is small relative to the dimension (see \cite{BIS} for more relevant examples).

(b) We also note that the Hilbert basis of the \ttt{9x9} ordered score sheets is very likely computable with Normaliz and the dual algorithm, however we have stopped this experiment after about two weeks since we were able by that time to formulate Theorem \ref{theo:HB}.

(c) At a first glimpse the computation times may seem paradoxical and
the reader may ask why parallelization has no effect
in the computations presented in Table \ref{times1}. We briefly explain this in the following.
On the one hand, when using the primal algorithm, Normaliz first finds the extreme rays of the cone $C_n=\RR_+\mathscr{M}_n$. This step is very fast and parallelization is very efficient. Then:
\begin{enumerate}
\item it computes the support hyperplanes of $C_n$ by using the well-known Fourier -- Motzkin elimination.
It starts from the zero cone, and then
the extreme rays are inserted successively (for details see \cite[Section 4]{BI});
\item intertwined with (1), a partial triangulation of C is builded (for details see \cite[Section 3.2]{BHIKS}). This partial triangulation is a subcomplex of the full lexicographic triangulation
obtained by inserting successively the extreme rays (for details see \cite[Section 4]{BI}).
\end{enumerate}
For all the examples in Table \ref{times1} we observed that:
\begin{enumerate}
\item the number of hyperplanes stays extremely low all the time during the Fourier -- Motzkin elimination;
\item the partial triangulation is empty! Note that this is the
case if and only if the full lexicographic triangulation is a unimodular triangulation.
\end{enumerate}
Parallelization as implemented in the current version of Normaliz is efficient if there are many support hyperplanes (at each step in the Fourier -- Motzkin elimination)
and after the (partial) triangulation reaches a reasonable size.

On the other hand, when using the dual algorithm the effective computation of the Hilbert basis is extremely fast and most of the time is spent on data transformation. This explains
why parallelization has almost no effect.
\end{remark}

\subsection{Multiplicity, Hilbert Series and the exploitation of symmetry}
For the following discussion we take as an example the monoid $\mathscr{M}_4$. For efficient computations we have encoded the entries of the $4x4$ score tables as in the table below.
\bigskip
$$
\begin{array}{|c c c c|}
\hline
\ast&x_1&x_2&x_3\\
x_4&\ast&x_5&x_6\\
x_7&x_8&\ast&x_9\\
x_{10}&x_{11}&x_{12}&\ast\\
\hline
\end{array}
$$
\bigskip
As a subcone of $\RR_+^{12}$, the cone   is
defined by the inequalities in Table \ref{ineq}.

\begin{table}[hbt]
\fbox{\parbox[b]{\linewidth}\centering
{\small \tabcolsep=1.3pt
\begin{tabular}{rrrrrrrrrrrrrr}
&$+x_1$&$+x_2$&$+x_3$&$-x_4$&$-x_5$&$-x_6$&      &      &      &         &         &         & $\ge 0$  \\
&      &      &      &$+x_4$&$+x_5$&$+x_6$&$-x_7$&$-x_8$&$-x_9$&         &         &         & $\ge 0$  \\
&      &      &      &      &      &      &$+x_7$&$+x_8$&$+x_9$&$-x_{10}$&$-x_{11}$&$-x_{12}$& $\ge 0$  \\
\end{tabular}}
}
\vspace*{2ex} \caption{Inequalities for
$C_4$}\label{ineq}
\end{table}

It is easy to notice the high degree of symmetry of these inequalities. As is often the case, it is better to use for effective computations the elegant approach of Sch\"urmann presented in \cite{Sch}.
If certain variables, for example $x_1,x_2,x_3$, occur
in all of the linear forms given in Table \ref{ineq}, then any permutation of them acts as a symmetry on
the corresponding cone, and the variables
$x_1,x_2,x_3$ may be replaced by their sum
$x_1+x_2+x_3$. (The cone has further
symmetries.) These substitutions may be used for a projection into
a space of much lower dimension, mapping the cone $C_4$ under
consideration to a cone $C'_4\subset \RR_+^{4}$. The computation of the Hilbert series of $\mathscr{M}_4=C_4\cap\ZZ^{12}$ may be replaced by the computation of the \emph{generalized Hilbert series} of $\mathscr{M'}_4=C'_4\cap\ZZ^{4}$ given by

$$
GH_{\mathscr{M'}_4}(t)=\sum_{x\in C'_4\cap\ZZ^4,\deg x=k } f(x)t^{k}.
$$

The theory of generalized Ehrhart functions has recently been
developed in several papers; see Baldoni, Berline, De Loera, K\"oppe, and Vergne \cite{BB1},
\cite{BB2}, and Schechter \cite{Sc} (they are also refereed as \emph{generalized Ehrhart series}, depending on the context). An extension of Normaliz to the computation of
generalized Hilbert series is presented in \cite{BS} and implemented as NmzIntegrate \cite{Nmz}. We briefly note that the approach to the computation of generalized Hilbert series in NmzIntegrate
is essentially based on \emph{Stanley decompositions} \cite{S}. For further details we point the reader to \cite{BS} and the manual of NmzIntegrate.

Table \ref{times2} contains the computation times for the
multiplicities of the ordered score sheets. The option \ttt{-v} indicates the standard Normaliz algorithm for computing multiplicities, and \ttt{-L} indicates the usage of NmzIntegrate for computing multiplicities.

\begin{table}[hbt]
\centering
\begin{small}
\begin{tabular}{|r|r|r|r|r|r|}\hline
\rule[-0.1ex]{0ex}{2.5ex}Input &\ttt{Nmz -v 1x}&\ttt{Nmz -v 20x}&\ttt{Nmz  -L 1x}&\ttt{Nmz -L 20x}\\ \hline
\strut \ttt{3x3}               & 0.003 s       & 0.014 s        & 0.003 s        & 0.014 s      \\ \hline
\strut \ttt{4x4}               & 0.045 s       & 0.082 s        & 0.004 s        & 0.015 s      \\ \hline
\strut \ttt{5x5}               & 1:29:41 h     & 9:49 m         & 0.011 s        & 0.021 s      \\ \hline
\strut \ttt{6x6}               & --            & --             & 0.161 s        & 0.183 s      \\ \hline
\strut \ttt{7x7}               & --            & --             & 8.513 s        & 9.329 s      \\ \hline
\strut \ttt{8x8}               & --            & --             & 10:10 m        & 10:44 m      \\ \hline
\end{tabular}
\vspace*{2ex} \caption{Computation times for multiplicities}\label{times2}
\end{small}
\end{table}

\begin{remark}\label{rem:tri}
For computing the multiplicity using the option \ttt{-v} Normaliz is producing a full triangulation of the cone. In all the cases in which we were able to compute the multiplicity by this algorithm,
the \emph{explicit} triangulation obtained was unimodular. As observed above, also the \emph{implicit} triangulations obtained in computations made using the option \ttt{-N} were unimodular. We also note that, in general, the triangulations made by Normaliz are \emph{not} the
pulling triangulations implied by Proposition \ref{prop:compressed}. This is a remarkable fact, and one may ask if it is true for all triangulations, or in other words: is the polytope $\mathscr{P}_n$  totally unimodular?
\end{remark}

Table \ref{times3} contains the computation times for the
Hilbert series of the ordered score sheets. The option \ttt{-q} indicates the standard Normaliz algorithm for computing Hilbert series, and \ttt{-E} indicates the usage of NmzIntegrate for computing Hilbert series.

\begin{table}[hbt]
\centering
\begin{small}
\begin{tabular}{|r|r|r|r|r|r|}\hline
\rule[-0.1ex]{0ex}{2.5ex}Input &\ttt{Nmz -q 1x}&\ttt{Nmz -q 20x}&\ttt{Nmz  -E 1x}&\ttt{Nmz -E 20x}\\ \hline
\strut \ttt{3x3}               & 0.004 s       & 0.015 s        & 0.005 s        & 0.017 s      \\ \hline
\strut \ttt{4x4}               & 0.099 s       & 0.083 s        & 0.025 s        & 0.042 s    \\ \hline
\strut \ttt{5x5}               & 6:54:03 h     & 34:11 m        & 0.378 s        & 0.466 s      \\ \hline
\strut \ttt{6x6}               & --            & --             & 12.853 s       & 15.296 s      \\ \hline
\strut \ttt{7x7}               & --            & --             & 8:04 m         & 9:45 m \\ \hline
\end{tabular}
\vspace*{2ex} \caption{Computation times for Hilbert series}\label{times3}
\end{small}
\end{table}

\begin{remark} The reader may wonder why there is no significant difference between the serial and the parallel computation times obtained in Tables \ref{times2} and \ref{times3} by NmzIntegrate. There is a simple explication for this.  NmzIntegrate uses the \emph{Stanley spaces} in the Stanley decomposition in order to parallelize computations, see \cite{BS}.  For the family of monoids $\mathscr{M'}_n$, the Stanley decomposition generated by Normaliz and used by NmzIntegrate contains only one Stanley space, so parallelization is not effective.
\end{remark}

\section*{Acknowledgments}

The authors are greatly indebted to Winfried Bruns for pointing out the guidelines in
proving Corollary \ref{cor:extremerays} and suggesting the content of subsection \ref{sec:triangulation}. We also wish to thank Sergiu Moroianu for several helpful comments. Moreover, the authors wish to thank the University of Porto (Portugal), in particular Manuel Delgado, for making it possible several discussions during the AMS-EMS-SPM International Meeting 2015. In addition, the first author wishes to thank also the University Jaume I of Castell\'on, where substantial parts of the paper were written, for the invitation and hospitality.


\begin{thebibliography}{99}

\bibitem{AI} V. Almendra, B. Ichim, jNormaliz. A graphical interface for Normaliz, available at \url{www.math.uos.de/normaliz}.

\bibitem{BB1} V. Baldoni, N. Berline, J.A. De Loera, M. K\"{o}ppe, M. Vergne, How to integrate a polynomial over a simplex, Math. Comp.~80~(2011)~297--325.

\bibitem{BB2} V. Baldoni, N. Berline, J.A. De Loera, M. K\"{o}ppe, M. Vergne, Computation of the highest coefficients of weighted Ehrhart quasi-polynomials of rational polyhedra, Found. Comp. Math.~12~(2012)~435--469.

\bibitem{BG} W.~Bruns, J.~Gubeladze, Polytopes, rings, and K-theory, Springer, 2009.

\bibitem{BHIKS} W.~Bruns, R.~Hemmecke, B. Ichim, M.~K\"oppe, C. S\"oger, Challenging computations of Hilbert bases of cones associated with algebraic statistics, Exp. Math.~20~(2011)~25--33.

\bibitem{BH} W.~Bruns and J.~Herzog, Cohen-Macaulay Rings, Rev. Ed., Studies in Advanced Mathematics, vol. 39, Cambridge University Press, Cambridge, 1996.

\bibitem{BI} W. Bruns, B. Ichim, Normaliz: Algorithms for affine monoids and rational cones, J. Algebra~324~(2010)~1098--1113.

\bibitem{Nmz} W. Bruns, B. Ichim, R. Sieg, T. R\"{o}mer, C. S\"{o}ger, Normaliz. Algorithms for rational cones and affine monoids, available at \url{www.math.uos.de/normaliz}.

\bibitem{BIS} W.Bruns, B. Ichim, C. S\"oger, The power of pyramid decompositions in Normaliz, J. Symb. Comp.~74~(2016)~513--536.

\bibitem{BK} W. Bruns, R. Koch, Computing the integral closure of an affine semigroup, Univ. Iagel. Acta Math.~39~(2001)~59--70.

\bibitem{BS} W. Bruns, C. S\"{o}ger, Generalized Ehrhart series and Integration in Normaliz, J. Symb. Comp.~68~(2015)~75--86.

\bibitem{E} R.~A.~Epstein, The theory of gambling and statistical logic, Revised ed. Academic Press, 1995. Second ed. by Springer Science \& Business Media, 2009.

\bibitem{EZ} Sh.~B.~Ekhad, D. Zeilberger, There are $\frac{1}{30}(r+1)(r+2)(2r+3)(r^2+3r+5)$ Ways For the Four Teams of a World Cup Group to Each Have $r$ Goals For and $r$ Goals Against, Preprint, arXiv:1407.1919 (2014).


\bibitem{HM} F. Harary, L. Moser, The Theory of Round-Robin Tournaments, Amer. Math. Monthly~73~(1966)~231--246.

\bibitem{K} G. Kendall, S. Knust, C. Ribeiro, S. Urrutia, Scheduling in sports: An annotated bibliography, Computers $\&$ Oper. Res.~37~(2010)~1--19.

\bibitem{OH} H. Ohsugi, T. Hibi, Convex polytopes all of whose reverse lexicographic initial ideals are squarefree, Proc. Am. Math. Soc.~129~(2001)~2541--2546.

\bibitem{Sc} M. Schechter, Integration over a polyhedron: an application of the Fourier--Motzkin elimination method, Amer. Math. Monthly~105~(1998)~246--251.

\bibitem{Sch} A. Sch\"{u}rmann, Exploiting polyhedral symmetries in social choice, Soc Choice Welf. ~40~(2013) ~1097--1110.


\bibitem{S1} R.~P.~Stanley, Decompositions of rational convex polytopes, Ann. Discrete Math. 6~(1980)~333--342.

\bibitem{S} R.~P.~Stanley, Linear Diophantine equations and local cohomology, Invent. Math. 68~(1982)~175--193.

\bibitem{St} R.~P.~Stanley, Enumerative Combinatorics 1, 2nd printing, Cambridge University Press, Cambridge, 1997.

\bibitem{Z} D.~Zeilberger, \url{www.math.rutgers.edu/~zeilberg/mamarim/mamarimhtml/worldcup.html}, personal~home~page.

\end{thebibliography}
\end{document}